\documentclass[10pt]{amsart}

\usepackage[utf8]{inputenc}
\usepackage{enumerate}
\usepackage{amsrefs}
\usepackage{dsfont}
\usepackage{amsmath}
\usepackage{amsthm}
\usepackage{amssymb}
\usepackage{amsfonts}
\usepackage{mathrsfs}  
\usepackage[shortlabels]{enumitem} 
\usepackage{color}
\usepackage{xcolor}
\usepackage[bookmarks=true,hyperindex,pdftex,colorlinks,citecolor=blue]{hyperref}
\usepackage{marginnote} 
\usepackage{graphicx} 
\usepackage{caption} 
\usepackage{soul} 
\usepackage{bm}
\usepackage[centertags]{amsmath}
\usepackage{tocvsec2}

\newcommand{\cof}{\text{cof}}

\newcommand{\rad}{\text{rad}}

\newcommand{\vep}{\varepsilon}
\newcommand{\ds}{\displaystyle}
\newcommand{\Sz}{\mathrm{Sz}}

\newcommand{\sN}{\ensuremath{\mathsf N}}

\renewcommand{\le}{\leqslant}
\renewcommand{\ge}{\geqslant}

\newcommand{\norm}[1]{\| #1\|}
\newcommand{\abs}[1]{| #1|}
\newcommand{\bnorm}[1]{\Big\| #1\Big\|}

\newcommand{\eps}{\varepsilon}

\newcommand{\Ndb}{\mathbb N}
\newcommand{\N}{\mathbb N}

\newcommand{\diam}{\text{diam}\,}

\theoremstyle{plain}
\newtheorem{theorem}{Theorem}[section]
\newtheorem{lemma}[theorem]{Lemma}
\newtheorem{corollary}[theorem]{Corollary}
\newtheorem{proposition}[theorem]{Proposition}

\theoremstyle{definition}
\newtheorem*{definition*}{Definition}

\newtheorem{example}[theorem]{Example}

\theoremstyle{remark}
\newtheorem{remark}[theorem]{Remark}

\title[]{An asymptotic analog of a local-to-global phenomenon for uniformly convex renormings}

\author{Florent P. Baudier}
\address{F. Baudier, Department of Mathematics, Texas A\&M University, College Station, TX 77843, USA.}
\email{florent@tamu.edu}

\author{Gilles Lancien}
\address{G. Lancien, Université de Franche-Comté, CNRS, LmB, F-25000 Besançon, France.}
\email{gilles.lancien@univ-fcomte.fr}


\thanks{F.~Baudier was supported by the National Science Foundation
  under Grant Number DMS-2055604.}

\begin{document}

\maketitle

\begin{abstract}
In this note, we investigate the renorming theory of Banach spaces with property $(\beta)$ of Rolewicz. In particular, we give a ``coordinate-free'' proof of the fact that every Banach space with property $(\beta)$ admits an equivalent norm that is asymptotically uniformly smooth; a result originally due to Kutzarova for spaces with a Schauder basis. We also show that if a natural modulus associated with a Banach space $X$ with property $(\beta)$ is positive at some point in the interval $(0,1)$, then $X$ admits an equivalent norm with property $(\beta)$. This is an asymptotic analog of a profound result from the local geometry of Banach spaces that states that if the modulus of uniform convexity of a Banach space $X$ is positive at some point in the interval $(0,2)$, then $X$ admits an equivalent norm that is uniformly convex.
\end{abstract}
\section{Introduction}
\label{sec:intro}

A Banach space $(X,\norm{\cdot})$ is uniformly convex if its modulus of uniform convexity $\delta_{\norm{\cdot}}\colon [0,2] \to X$, defined as
\begin{equation}
\delta_{\norm{\cdot}}(t):= \inf \left\{ 1- \bnorm{\frac{x+y}{2}} \colon \max\{\norm{x},\norm{y}\}\le 1 \textrm{ and } \norm{x-y}\ge t\right\},
\end{equation}
is positive on the interval $(0,2]$.
In \cite{James1964}, James introduced the notion of uniformly non-squareness and proved that every uniformly non-square Banach space is super-reflexive. It is elementary to verify that a Banach space $(X,\norm{\cdot})$ is uniformly non-square if and only if $\delta_{\norm{\cdot}}(t_0)>0$ for some $t_0\in(0,2)$. Therefore, combining James' result with Enflo's renorming theorem \cite{Enflo1972}, which states that every super-reflexive Banach space admits an equivalent norm that is uniformly convex, we obtain the following remarkable renorming result.
\begin{theorem}
\label{thm:James-Enflo}
Let $(X,\norm{\cdot})$ be a Banach space and assume that $\delta_{\norm{\cdot}}(t_0)>0$ for some $t_0\in(0,2)$. Then, $X$ admits an equivalent norm $\abs{\cdot}$ such that $\delta_{\abs{\cdot}}(t)>0$ for all $t\in (0,2]$. 
\end{theorem}

It is easy to see that $\delta_{\norm{\cdot}}(t)>0$ for all $t\in(0,2]$ if and only if $\beta_{\norm{\cdot}}(t)>0$ for all $t\in(0,2]$, where
\begin{align*}
&\beta_{\norm{\cdot}}(t):=\\
&\inf \left\{ \max_{i\in\{1,2\}} \left\{1- \bnorm{\frac{x+y_i}{2}} \right\} \colon \max\{\norm{x},\norm{y_1}, \norm{y_2}\}\le 1 \textrm{ and } \norm{y_1-y_2}\ge t\right\},
\end{align*}
A natural asymptotic analog of the notion of uniform convexity can be defined in terms of the asymptotic $(\beta)$-modulus $\bar{\beta}_{\norm{\cdot}} \colon [0,1] \to [0,\infty]$, defined as
\begin{align*}
&\bar{\beta}_{\norm{\cdot}}(t):=\\
&\inf \left\{\sup_{n\in \N} \left\{ 1- \bnorm{\frac{x+y_n}{2}} \right\} \colon \max\{\norm{x}, \sup_{n\in \N}\norm{y_n}\}\le 1 \textrm{ and } \inf_{n\neq m}{\norm{y_m-y_n}}\ge t\right\}.
\end{align*}

In \cite{Rolewicz1987}, Rolewicz introduced property $(\beta)$ and Kutzarova showed in \cite[Theorem 7]{Kutzarova1991} that a Banach space has Rolewicz's property $(\beta)$ if and only if $\bar{\beta}_{\norm{\cdot}}(t)>0$ for all $t\in(0,1]$.

The main result of this note is the following asymptotic analog of Theorem \ref{thm:James-Enflo}.

\begin{theorem}
\label{thm:main}
Let $(X,\norm{\cdot})$ be a Banach space and assume that $\bar{\beta}_{\norm{\cdot}}(t_0)>0$ for some $t_0\in(0,1)$. Then, $X$ admits an equivalent norm $\abs{\cdot}$ such that $\bar{\beta}_{\abs{\cdot}}(t)>0$ for all $t\in (0,1]$.
\end{theorem}

Property $(\beta)$ is a strengthening of another natural asymptotic analog of uniform convexity, namely, asymptotic uniform convexity.
The modulus of asymptotic uniform convexity $\bar{\delta}_{\norm{\cdot}}\colon [0,\infty) \to X$, originally introduced by Milman \cite{Milman1971}, is defined as \begin{equation}
    \bar{\delta}_{\norm{\cdot}}(t)=\inf_{\stackrel{x\in X}{\norm{x}=1}}\sup_{Y\in \cof(X)}\inf_{\stackrel{y\in Y}{\norm{y}=1}}\norm{x+t y}-1,
\end{equation}
where $\cof(X)$ is the set of all closed finite-codimensional subspaces of $X$. Note that $\bar{\delta}_{\norm{\cdot}}(t)>0$ for all $t\in(2,\infty)$ and the norm of $X$ is said to be asymptotically uniformly convex if $\bar{\delta}_{\norm{\cdot}}$ remains positive on $(0,2]$.

It is well-known that $\bar{\delta}_{\norm{\cdot}}(t)\ge \bar{\beta}_{\norm{\cdot}}(\frac{t}{2})$ for all $t\in(0,1)$, and hence, if a Banach space has property $(\beta)$, then it is asymptotically uniformly convex (see \cite{Huff1980} or \cite[Lemma 4.0.2]{DKR2016}). To prove our main theorem, we will need a slight refinement of the comparison between the modulus of asymptotic uniform convexity and the $(\beta)$-modulus (see Proposition \ref{prop:beta-delta}).

According to the current state-of-the-art, to renorm a space with property $(\beta)$ it is sufficient to show that there is an equivalent norm that is asymptotically uniformly convex and an equivalent norm that is asymptotically uniformly smooth. Since Banach spaces with property $(\beta)$ are automatically reflexive (see Proposition \ref{prop:beta-reflexive} for another refinement) it follows from classical techniques that there is an equivalent norm with property $(\beta)$ (see for instance \cite[Theorem 4.2 and Remark 4.5]{DKLR2017}).

The modulus of asymptotic uniform smoothness $\bar{\rho}_{\norm{\cdot}}\colon [0,\infty) \to [0,\infty)$, also originally from \cite{Milman1971}, is defined as \begin{equation}
    \bar{\rho}_{\norm{\cdot}}(t)=\sup_{\stackrel{x\in X}{\norm{x}=1}}\inf_{Y\in \cof(X)}\sup_{\stackrel{y\in Y}{\norm{y}=1}}\norm{x+t y}-1,
\end{equation}
and the norm of $X$ is said to be asymptotically uniformly smooth if $\ds{\lim_{t\to 0}}\frac{\bar{\rho}_{\norm{\cdot}}(t)}{t}=0$.
It follows from Kutzarova's example \cite[Example 5]{Kutzarova1990} of a uniformly convex Banach space that is not nearly uniformly smooth that Property $(\beta)$ does not necessarily imply asymptotic uniform smoothness. However, building on the work of Pruss \cite{Prus1989}, Kutzarova showed in \cite{Kutzarova1990} that every Banach space with property $(\beta)$ (and with a Schauder basis) admits an equivalent norm that is asymptotically uniformly smooth.
The Schauder basis assumption can be lifted by invoking the Szlenk index theory and its deep connection with asymptotically uniformly smooth renormings. This will be explained in more detail in Section \ref{sec:AUS}.
In fact, Kutzarova's argument shows that the condition $\bar{\beta}_{\norm{\cdot}}(t_0)>0$ for some $t_0\in (0,1)$ is sufficient in order to construct an equivalent asymptotic uniformly smooth norm. 
In Section \ref{sec:AUS}, we provide a new proof of this fact which works for Banach spaces that do not need to possess a Schauder basis nor be separable. The argument that we use can be seen as a ``coordinate-free'' asymptotic analog of a technique, due originally to Gurarii-Gurarii \cite{GurariiGurarii1971} and also used in \cite{Kutzarova1990}, which produces upper-$p$-estimates for basic sequences in uniformly convex spaces.
In order to work in the most general framework, we manipulate weakly-null trees indexed by directed nets and rely on the asymptotic uniform smoothness renorming theory via upper estimates on branches of weakly-null trees, as originally done in \cite{GKL2001} in the context of separable spaces and countably branching trees and further developed by Causey in \cite{Causey2018a, Causey2018b, Causey2019, Causey3.5} for arbitrary Banach spaces and trees indexed by directed sets. This coordinate-free approach to asymptotic renormings is an asymptotic variant of Pisier's renorming technique \cite{Pisier1975}. It is worth pointing out that there is another approach to asymptotic renorming developed by Odell and Schlumprecht (see \cite{KOS1999}, \cite{OdellSchlumprecht2002}, and \cite{OdellSchlumprecht2006RACSAM}) which requires first to embed the space to be renormed into a space with a coordinate system.
In Section \ref{sec:back}, we recall the technical background that will allow us to work in the most general framework.
In Section \ref{sec:AUS}, we prove the asymptotically smooth renorming theorem under a pointed positivity condition of the $(\beta)$-modulus.
Finally, in Section \ref{sec:James}, we complete the proof of Theorem \ref{thm:main}.
\vskip .3cm
\noindent {\bf Acknowledgement.} This work is dedicated to Gilles Godefroy. The present note addresses questions in renorming theory and asymptotic geometry of Banach spaces, two of the many directions of the geometry of Banach spaces in which Gilles Godefroy had a tremendous influence. 

\section{Technical background}
\label{sec:back}
To handle arbitrary Banach spaces we work with weakly null trees indexed over directed sets. We recall the definition of this type of trees and set the corresponding notation. Given an arbitrary set $D$, let:
\begin{itemize}
    \item $\emptyset$ be the empty sequence,
    \item $D^{\le n}:=\{\emptyset\}\cup \cup_{i=1}^n D^i$ be the set of all finite sequences in $D$ of length at most $n$,
    \item $D^{<\omega}:=\cup_{n=1}^\infty D^{\le n}$ be the set of all finite sequences in $D$,
    \item $D^\omega$, the set of all infinite sequences in $D$,
    \item $D^{\le \omega}:=D^{<\omega}\cup D^\omega$ be the set of all (finite or infinite) sequences in $D$.
\end{itemize}
For $s,t\in D^{<\omega}$, we let $s\frown t$ denote the concatenation of $s$ with $t$, and we simply write $s\frown a$ instead of $s\frown (a)$, whenever $a\in D$. We let $|t|$ denote the length of $t$.  For $0\le i\le|t|$, we let $t_{\restriction i}$ denote the initial segment of $t$ having length $i$, where $t_{\restriction 0}=\emptyset$. If $s\in D^{<\omega}$, we let $s\prec t$ denote the relation that $s$ is a proper initial segment of $t$, i.e., $s=t_{\restriction i}$ for some $0\le i<|t|$. We denote $s \preceq t$ if $s\prec t$ or $s=t$.  It is well known that  $(D^{<\omega}, \prec)$ is a tree in the set-theoretic sense.

If $(D, \leqslant)$ is a directed set and $(x_t)_{t\in D^{<\omega}}\subset X$, we say $(x_t)_{t\in D^{<\omega}}$ is  a \emph{weakly null tree} in $X$, provided that for each $t\in D^{<\omega}$, $(x_{t\frown s})_{s\in D}$ is a weakly null net. A function $\varphi \colon D^{<\omega}\to D^{<\omega}$ is said to be a \emph{pruning} provided that \begin{enumerate}[(i)]
\item $\phi$ preserves the tree ordering, i.e., if $s\prec t$, then $\varphi(s)\prec \varphi(t)$,
\item $\phi$ preserves the length, i.e., $|\varphi(t)|=|t|$ for all $t\in D^{<\omega}$, and moreover 
\item[(ii')] if $\varphi((t_1, \ldots, t_k))=(s_1, \ldots, s_k)$, then $t_i\leqslant s_i$ for all $i\in \{1,\dots, k\}$. 
\end{enumerate} 
We define prunings $\varphi:D^{\leqslant n}\to D^{\leqslant n}$ similarly. One reason one might want to prune a tree is to stabilize maps that are defined on the leaves of trees as it can be seen in the next classical lemma, that we prove for the comfort of the reader.   

\begin{lemma}
\label{lem:pruning1}
Let $(D,\leqslant)$ be a directed set with no upper bound, $F$ a finite set, $n\in \Ndb$, and $f \colon D^n \to F$ a function. Then, there exists a pruning $\varphi \colon D^{\le n}\to D^{\le n}$ such that $f \circ \varphi$ is constant on $D^n$.
\end{lemma}

\begin{proof} We work by induction on $n$. So assume first that $n=1$. For $x\in F$, denote $I_x=\{t\in D \colon f(t)=x\}$ and observe that $\cup_{x\in F}I_x=D$. Since $F$ is finite and $(D,\leqslant)$ has no upper bound, there must exist some $x\in F$ such that $I_x$ is cofinal in $D$. This means that for any $t\in D$, there exists $s_t\in I_x$ such that $t\leqslant s_t$. Define $\varphi(t)=s_t$ and $\varphi(\emptyset)=\emptyset$. We have that $\varphi \colon D^{\le 1}\to D^{\le 1}$ is a pruning and $f \circ \varphi=x$ on $D$.

Assume now that the result holds for $n\in \Ndb$ and let $f\colon D^{n+1}\to F$. For each $t\in D$, define $f_t\colon D^n\to F$ by $f_t(t_1, \ldots, t_n)=f(t, t_1, \ldots, t_n)$.  By the inductive hypothesis, there exists a pruning $\varphi_t\colon D^{\le n}\to D^{\le n}$ and $x_t\in F$ such that $f_t\circ \varphi_t= x_t$ on $D^n$. Define $g\colon D^{1}\to F$ by $g((t))=x_t$. By the base case, there exists a pruning $\psi \colon D^{\le 1}\to D^{\le 1}$ such that $g \circ \psi$ is constant on $D$.  Define $\varphi \colon D^{\le n+1}\to D^{\le n+1}$ by, $\varphi(\emptyset)=\emptyset$, $\varphi(t)=\psi(t)$, and $\varphi(t, t_1, \ldots, t_k) = \psi(t)\smallfrown \varphi_{\psi(t)}(t_1, \ldots, t_k)$, for $1\le k\le n$. It is straightforward to verify that $\varphi$ is a pruning.
\end{proof}

The following corollary is then immediate.

\begin{corollary}
\label{cor:pruning2} Let $(D,\leqslant)$ be a directed set with no upper bound, $(K,d)$ a totally bounded metric space, $n\in \Ndb$, and $f\colon D^n \to K$ a function. Then, for any $\eps>0$, there exists a pruning $\varphi \colon D^{\le n}\to D^{\le n}$ and a subset $B$ of $K$ with $\diam(B)\le \vep$ such that $(f \circ \varphi)(t) \in B$ for all $t\in D^n$.
\end{corollary}

The general pruning lemma will be instrumental in showing that a Banach space belongs to a class of Banach spaces defined in terms of certain upper estimates on branches of weakly null trees indexed over weak neighborhood bases of $0$ (directed by reverse inclusion). For a Banach space $X$, we denote $S_X$ its unit sphere and $B_X$ its closed unit ball. Given $1<p\le \infty$, the class $\sN_p$, introduced by Causey in \cite{Causey3.5}, is the collection of Banach spaces $X$ for which there exists a constant $c>0$ such that for any weak neighborhood basis $D$ at $0$ in $X$, any $n\in \Ndb$ and any weakly null tree $(x_t)_{t\in D^{\le n}}\subset S_X$, there exists $\tau\in D^n$ such that $\|\sum_{i=1}^n x_{\tau|_i}\|\le cn^{1/p}$. 

The theory of asymptotically uniformly smooth renormings and classes such as the $\sN_p$-classes are intimately connected to an ordinal index introduced by Szlenk \cite{Szlenk1968} for other purposes. We recall its definition.

Let $X$ be a Banach space and $K$ be a weak$^*$-compact subset of $ X^*$. For each $\eps>0$, the Szlenk derivation operation is defined by 
\begin{align*}
    s_\eps(K) &:= K\setminus\{V\subset X^*\colon V\text{ weak$^*$-open and }\mathrm{diam}(V\cap K)\le \eps\}\\
              & \, = \{ x^*\in K \colon \forall V\in \mathcal{V}_{w^*}(x^*)\, \mathrm{diam}(V\cap K)> \eps\},
\end{align*}
where $\mathcal{V}_{w^*}(x^*)$ denotes the set of all weak$^*$-neighborhoods of $X^*$. Given an ordinal $\xi$, the derived set of order $\xi$, denoted $s_\eps^\xi(K)$, is defined inductively by letting 
\begin{itemize}
    \item $s_\eps^0(K):=K$,  
    \item $s_\eps^{\xi+1}(K):=s_\eps(s^\xi_\eps(K))$,
    \item $s_\eps^{\xi}(K):=\cap_{\zeta<\xi}s^\zeta_\eps(K)$ if $\xi$ is a limit ordinal. 
\end{itemize} 
 We also use the obvious convention $s_0^\xi(K)=K$, for any ordinal $\xi$. The ordinal $\Sz(X,\eps)$ is defined as the least ordinal $\xi$ so that $s_\eps^\xi(B_{X^*})=\emptyset$, if such ordinal exists, and  we denote $\Sz(X,\eps)=\infty$ otherwise, with the obvious convention that $\infty$ is larger than any ordinal. The \emph{Szlenk index of $X$} is the ordinal number defined by
\begin{equation*}
    \Sz(X) := \sup_{\eps>0}\Sz(X,\eps).
\end{equation*}

In the next theorem, we state two characterizations, that will be needed in the sequel, of the class of Banach spaces that are asymptotically uniformly smoothable (see \cite{KOS1999}, \cite{GKL2001}, and \cites{Causey2018a,Causey3.5}). 

\begin{theorem}\label{thm:AUS}
    Let $X$ be a Banach space. The following assertions are equivalent.
    \begin{enumerate}
        \item $X$ admits an equivalent asymptotically uniformly smooth norm.
        \item $X\in \sN_p$ for some $p\in(1,\infty]$.
        \item $\Sz(X)\le \omega$.
    \end{enumerate}
\end{theorem}

\section{Property $(\beta)$ and asymptotically uniformly smooth renormings}
\label{sec:AUS}
Recall that a Banach space with property $(\beta)$ is automatically reflexive and asymptotically uniformly convex. However, we already mentioned in the introduction that there are examples of Banach spaces with property $(\beta)$ that are not asymptotically uniformly smooth. Nevertheless, it is known that a Banach space with property $(\beta)$ admits an equivalent norm that is asymptotically uniformly smooth.

This result goes back to a paper by D. Kutzarova \cite{Kutzarova1990}, where it was proved that a Banach space with a Schauder basis and property $(\beta)$ has an equivalent \emph{nearly uniformly smooth} norm.
The nearly uniformly smooth property, which we shall not define here, was independently introduced by Sekowski and Stachura \cite{SekowskiStachura1988} and S. Prus \cite{Prus1989} and it was later observed that a Banach space is nearly uniformly smooth if and only if it is reflexive and asymptotically uniformly smooth.

Kutzarova's renorming result in the general case (separable or not) follows for instance from the fact that a reflexive Banach space $X$ satisfies $\mathrm{Sz}(X)\le \omega$ if and only if all its subspaces with a Schauder basis do (see \cite{DKLR2017} for the proof which is a refinement of an original argument from \cite{Lancien1996}). 

Thus, the following theorem follows from Kuztarova's argument and the discussion above.

\begin{theorem}
\label{thm:beta->AUS}
Let $(X,\norm{\cdot})$ be a Banach space. If there exists $t_0\in(0,1)$ such that 
$\bar{\beta}_{\norm{\cdot}}(t_0)>0$, then $X$ admits an equivalent norm that is asymptotically uniformly smooth.
\end{theorem}

The proof of Theorem \ref{thm:beta->AUS} that we are about to give does not rely on a coordinatization step and the fact that having Szlenk index at most $\omega$ is determined, in the reflexive case, by Schauder basic sequences. Instead, we provide in this note a ``coordinate-free'' version of Kutzarova's argument, and we prove that a Banach space with property $(\beta)$ belongs to $\sN_p$ for some $p>1$. The conclusion will then follow from Theorem \ref{thm:AUS} recalled in Section \ref{sec:back}. 

\begin{remark}\label{rem:Mazur}
It is a well-known application of Mazur's lemma that for every $\varepsilon>0$ and every normalized weakly-null sequence in an infinite-dimensional Banach space, one can extract a subsequence that is $(1+\varepsilon)$-basic. A similar argument together with Lemma \ref{lem:pruning1}, insures that if $(x_t)_{t\in D^{\le n}}$ is a normalized weakly null tree (where $D$ is a weak neighborhood basis of $0$), then there exists a pruning $\varphi \colon D^{\le n} \to D^{\le n}$ such that $(x_{\varphi(\tau)_{\restriction 1}},\ldots,x_{\varphi(\tau)_{\restriction n}})$ is $(1+\varepsilon)$-basic for all $\tau \in D^n$.   
\end{remark}

A key idea in Kutzarova's argument leading up to Theorem \ref{thm:beta->AUS}, and which was also used by Prus \cite{Prus1989}, is a beautiful technique from N.I. Gurarii and V.I. Gurarii \cite{GurariiGurarii1971} which produces non-trivial upper estimates for basic sequences in uniformly convex spaces. The following lemma can be seen as an asymptotic and coordinate-free analog of the Gurarii-Gurarii technique.

\begin{lemma}
\label{lem:beta->AUS} 
Let $(X,\norm{\cdot})$ be a Banach space and assume that $\bar{\beta}_{\norm{\cdot}}(t_0)>0$ for some $t_0\in(0,1)$. Then there exists $\delta \in (0,1)$ such that for any $n\ge 2$, any weakly null tree $(x_t)_{t\in D^{\le n}}$ in $S_X$, where $D$ is any weak neighborhood basis of $0$ in $X$, there exists a pruning $\psi\colon D^{\le n} \to D^{\le n}$ such that for all $\tau \in D^n$, all $1\le i< n$, the following are well defined:
$$u^i_{\psi(\tau)}=\sum_{j=1}^ix_{\psi(\tau)_{\restriction j}},\ y^i_{\psi(\tau)}=\frac{u^i_{\psi(\tau)}}{\|u^i_{\psi(\tau)}\|},\ v^i_{\psi(\tau)}=\sum_{j=i+1}^nx_{\psi(\tau)_{\restriction j}},\ z^i_{\psi(\tau)}=\frac{v^i_{\psi(\tau)}}{\|v^i_{\psi(\tau)}\|}$$
and
$$\|y^i_{\psi(\tau)}+z^i_{\psi(\tau)}\|\le 2-2\delta.$$
\end{lemma}

\begin{proof}
Since $\bar{\beta}_{\norm{\cdot}}(t_0)>0$, there is $\delta\in(0,1)$ such that for all $x, z_1,z_2,\ldots$ in $S_X$, if $\inf_{i\neq j} \|{z_i-z_j}\|\ge t_0$ then $\|{x+z_{i_0}}\|\le 2-2\delta$ for some $i_0\in \Ndb$.

Let $D$ be a weak neighborhood basis of $0$ in $X$, $n\ge 2$, and $(x_t)_{t\in D^{\le n}}$ a weakly null tree in $S_X$. It is clearly enough to find a pruning satisfying the conclusion of the lemma for a given $i\in \{1,\ldots,n-1\}$. So, let us fix $1\le i <n$. After a first pruning, we may assume by  Remark \ref{rem:Mazur} that $(x_{\tau_{\restriction 1}},\ldots,x_{\tau_{\restriction n}})$ is $2$-basic for all $\tau \in D^n$. In particular $\|u_\tau^i\|\ge \frac12$ and $\|v_\tau^i\|\ge \frac12$, for all $\tau \in D^n$. Before to proceed with the heart of the proof, we recall that, by Lemma \ref{lem:pruning1}, it is enough to find one branch $\tau \in D^n$ such that $\|y^i_{\tau}+z^i_{\tau}\|\le 2-2\delta.$

Let us fix $\sigma \in D^i$. We will build inductively $\tau_k=\sigma \smallfrown \sigma_k \in D^n$, for $k\in \Ndb$ such that the sequence $(z_{\tau_k}^i)_{k=1}^\infty$ is $t_0$-separated. Pick any $\sigma_1 \in D^{n-i}$ and set $\tau_1=\sigma \smallfrown \sigma_1$. Assume that $\sigma_1,\ldots,\sigma_k \in D^{n-i}$ have been constructed. Since $t_0\in (0,1)$, we can find $\eta>0$ so that $1-2n\eta>t_0$. For each $l\in \{1,\ldots,k\}$, pick $x^*_l\in S_{X^*}$ such that $x^*_l(z_{\tau_l}^i)=1$. Let now $V=\{x\in X,\ |x^*_l(x)|<\eta,\ \text{for all}\ 1\le l \le k\}$. Since $V$ is a weak neighborhood of $0$ and our tree is weakly null, we can pick recursively $U_{i+1},\ldots, U_n \in D$ so that for all $j\ge i+1$ $x_{\sigma\smallfrown (U_{i+1},\ldots, U_j)}\in V$ and we set $\tau_{k+1}=\sigma\smallfrown (U_{i+1},\ldots, U_n)$. Then for all $1\le l \le k$, $|x^*_l(v^i_{\tau_{k+1}})|\le n\eta$ and $|x^*_l(z^i_{\tau_{k+1}})|\le 2n\eta$. It follows that for all $1\le l \le k$, $\|z^i_{\tau_{k+1}}-z^i_{\tau_{l}}\|>t_0$. This finishes the inductive construction of $(\tau_k)_{k=1}^\infty$. We now note that for all $k\in \N$, $y_{\tau_k}^i=\big(\sum_{j=1}^i x_{\sigma_{\restriction j}}\big)\|\sum_{j=1}^i x_{\sigma_{\restriction j}}\|^{-1}:=x$. Then, it follows from the definition of the $(\beta)$-modulus recalled at the beginning of the proof that there exists $k_0\in \Ndb$ such that $\|x+z_{\tau_{k_0}}^i\|=\|y_{\tau_{k_0}}^i+z_{\tau_{k_0}}^i\|\le 2-2\delta$. This concludes the proof of this lemma.
\end{proof}

We shall also need the following elementary fact.
\begin{lemma}\label{lem:beta+calculus} Let $\delta \in (0,1)$. Then, there exists $p>1$ and $\nu \in (0,\frac12)$ such that whenever $y,z \in S_X$ satisfy $\|y+z\|\le 2-2\delta$, then, $\|y+tz\|^p\le 1+t^p$, for all $t\in [1-\nu,1+\nu]$.
\end{lemma}

\begin{proof}
There exists $p>1$ such that $(2-\delta)^p<2$. Then, the functions $f:t\mapsto (1+t-\delta)^p$ and $g:t\mapsto 1+t^p$ are continuous on $[0,\infty)$ and satisfy $f(1)<g(1)$. So, there exists $\nu \in (0,\frac12)$ such $f(t)\le g(t)$, for all $t\in [1-\nu,1+\nu]$. Let now $y,z \in S_X$ such that $\|y+z\|\le 2-2\delta$, then, for all $t\in [1-\nu,1+\nu]$:
$$\|y + tz\|^p \le \Big(\frac12\|{y}\| + \big(t-\frac12\big)\|{z}\| + \frac12 \|{ y + z}\|\Big)^p \le (1+t-\delta)^p\le 1+t^p.$$
\end{proof}

We are now ready to prove the theorem below which immediately implies Theorem \ref{thm:beta->AUS} via Theorem \ref{thm:AUS}.

\begin{theorem}
Let $(X,\norm{\cdot})$ be a Banach space. If there exists $t_0\in(0,1)$ such that 
$\bar{\beta}_{\norm{\cdot}}(t_0)>0$, then there exists $p\in(1,\infty)$ such that $X\in \sN_p$.
\end{theorem}

\begin{proof}
Let $\delta \in (0,1)$ be given by Lemma \ref{lem:beta->AUS}. Let $p>1$ and $\nu \in (0,\frac12)$ be the constants associated with $\delta$, through Lemma \ref{lem:beta+calculus}. We set $C=\frac{3}{\nu}$. We will show by induction on $n$, that for any weak neighborhood basis $D$ at $0$ in $X$, any $n\in \N$ and any $(x_t)_t\in D^{\le n}$ weakly null tree in $S_X$, there exists $\tau \in D^n$ such that  $\|\sum_{j=1}^n x_{\tau_{\restriction j}}\|\le Cn^{1/p}$. Since $C\ge 3$, this is clearly true for $n\in \{1,2,3\}$. So assume $n>3$ and that our induction hypothesis is true for all $i\in\{1,\ldots,n-1\}$. Consider $D$ a weak neighborhood basis of $0$ in $X$ and $(x_t)_{t\in D^{\le n}}$ a weakly null tree in $S_X$. For $1\le i<n$ and $\tau \in D^n$, we adopt the notation $u_\tau^i,y_\tau^i,v_\tau^i,z_\tau^i$ from Lemma \ref{lem:beta->AUS}, that we complete with $u^0_\tau=v^n_\tau=0$ and $v^0_\tau=u^n_\tau=\sum_{j=1}^nx_{\tau_{\restriction j}}$

It clearly follows from our induction hypothesis and the pruning Lemma \ref{lem:pruning1} that, after pruning, we may assume that for all $1\le i<n$ and all $\tau \in D^n$, $\|u_\tau^i\|\le Ci^{1/p}$ and $\|v_\tau^i\|\le C(n-i)^{1/p}$. Next, we use Lemma \ref{lem:beta->AUS} to justify that, again after pruning, we may assume that for all $1\le i<n$ and all $\tau \in D^n$, $\|y_\tau^i+z_\tau^i\|\le 2-2\delta$. 

Fix now $\tau \in D^n$. The proof will be complete if we show that there exists $1\le i<n$ such that $\|u_\tau^i+v_\tau^i\|\le Cn^{1/p}$. Note that $\|u^0_\tau\|<\|v^0_\tau\|$, $\|u^n_\tau\|>\|v^n_\tau\|$ and for all $0\le i<n$, $\big|\|u^{i+1}_\tau\|-\|u^{i}_\tau\|\big|\le 1$ and $\big|\|v^{i+1}_\tau\|-\|v^{i}_\tau\|\big|\le 1$. It is an easy exercise to check that it implies the existence of $0\le i_0\le n$ such that  $\big|\|u^{i_0}_\tau\|-\|v^{i_0}_\tau\|\big|\le 1$. Assume first that $\|u^{i_0}_\tau\|\le \frac{C}{3}$ or $\|v^{i_0}_\tau\|\le \frac{C}{3}$. Then $\|u_\tau^{i_0}+v_\tau^{i_0}\|\le \frac{2C}{3}+1 \le C \le Cn^{1/p}$. So we can assume that $\|u^{i_0}_\tau\|> \frac{C}{3}$ and $\|v^{i_0}_\tau\|> \frac{C}{3}$, which implies that $1\le i_0<n$, and also that $(1-\nu)\|u^{i_0}_\tau\|\le \|v^{i_0}_\tau\| \le (1+\nu)\|u^{i_0}_\tau\|$. It now follows from the fact that $\|y_\tau^{i_0}+z_\tau^{i_0}\|\le 2-2\delta$, Lemma \ref{lem:beta+calculus}, and a homogeneity argument that 
$$\|u_\tau^{i_0}+v_\tau^{i_0}\|^p\le \|u_\tau^{i_0}\|^p+\|v_\tau^{i_0}\|^p \le C^p\big(i_0+(n-i_0)\big)=C^pn.$$
This finishes our inductive proof. 
\end{proof}

\begin{remark} 
We have chosen to present a proof using the general pruning Lemma \ref{lem:pruning1} and trees indexed by weak neighborhood bases directed by reverse inclusion. Alternatively, we could have used a separable reduction argument and thus work with countably branching trees and use the classical Ramsey Theorem for colorings in $[\Ndb]^k$ in place of the general pruning lemma. 
\end{remark}

\section{A local-to-global phenomenon for $(\beta)$-renormings}

\label{sec:James}
In this section, we complete the proof of our main theorem. Let us recall the strategy to prove Theorem \ref{thm:main}. Given a Banach space $(X,\norm{\cdot})$, if $\bar{\beta}_{\norm{\cdot}}(t_0)>0$ for some $t_0\in(0,1)$, then it will be sufficient to show that $X$ is reflexive, admits an equivalent norm that is asymptotically uniformly convex, and admits an equivalent norm that is asymptotically uniformly smooth. The existence of an equivalent norm that is asymptotically uniformly smooth was taken care of in Section \ref{sec:AUS}. As for reflexivity, it was known already to Rolewicz \cite{Rolewicz1987} that Banach spaces with property $(\beta)$ were reflexive and the refinement below was already observed by Kutzarova \cite{Kutzarova1990}. For the convenience of the reader, we provide an elementary argument that is different from the one alluded to in \cite{Kutzarova1990} and which relies on the notion of nearly uniform convexity and a result of Landes \cite{Landes1989}. 

\begin{proposition}
\label{prop:beta-reflexive}
Let $(X,\norm{\cdot})$ be a Banach space. If there exists $t_0\in (0,1)$ such that $\bar{\beta}_{\norm{\cdot}}(t_0)>0$, then $(X,\norm{\cdot})$ is reflexive.   
\end{proposition}

\begin{proof}
Let $(X,\norm{\cdot})$ be a non-reflexive Banach space. Then, it follows from James' criterion that for any $t \in (0,1)$, there exist sequences $(x_n)_{n=1}^\infty$ in $S_X$ and $(x_n^*)_{n=1}^\infty$ in $S_{X^*}$ such that $x_n^*(x_i)>t$ for all $i \ge n$ and $x_n^*(x_i)=0$ for all $ i < n$. This clearly implies that $(x_n)_{n=1}^\infty$ is $t$-separated, while $\|x_1+x_n\|\ge 2t$ for all $n\in \N$. It follows that $\bar{\beta}_{\norm{\cdot}}(t)\le 1-t$, for all $t\in (0,1)$. Since $\bar{\beta}_{\norm{\cdot}}$ is increasing, this implies that $\bar{\beta}_{\norm{\cdot}}(t)=0$, for all $t\in (0,1)$. 
\end{proof}

It remains to show the existence of an equivalent norm that is asymptotically uniformly convex. This can be done using Szlenk index techniques together with the duality between asymptotic uniform smoothness and asymptotic uniform convexity. By piecing out various results from the literature one could prove a partial version of Theorem \ref{thm:main} where $t_0$ would be restricted to $(0,\frac14)$. To emphasize and clarify the modifications that are needed to handle the entire open unit interval we describe the general argument and recall some useful results in the process.
Let us first discuss the perfect duality between the asymptotic uniform smoothness of $X$ and the weak$^*$-asymptotic uniform convexity of the dual $X^*$. If we denote by $\cof^*(X^*)$ the set of all weak$^*$-closed subspaces of
$X^*$ of finite codimension, we say that $(X^*, \norm{\cdot}^*)$ is \emph{weak$^*$-asymptotically uniformly convex} if $\bar{\delta}^*_{\norm{\cdot}}(t)>0$ for all $t>0$ where the modulus of weak$^*$-asymptotic uniform convexity $\bar{\delta}^*_{\norm{\cdot}}\colon [0,\infty) \to [0,\infty)$ is defined as
\begin{equation}
\bar{\delta}^*_{\norm{\cdot}}(t)=\inf_{\norm{x^*}^*=1}\sup_{E\in \cof^*(X^*)}\inf_{\norm{y^*}^*=1} \norm{ x^*+ty^*}^*-1,
\end{equation}
with $\norm{\cdot}^*$ denoting the canonical dual norm.
\begin{remark}
Obviously, when $X$ is reflexive, $\cof^*(X^*)=\cof(X^*)$ and 
\begin{equation}
    \bar{\delta}^*_{\norm{\cdot}} = \bar{\delta}_{\norm{\cdot}^*}.
\end{equation}
The space $\ell_1$ is AUC and AUC$^*$ for the weak$^*$ topology coming from $c_0$. However, it can be shown that $\ell_1$ admits isometric preduals such that it is not AUC$^*$ for the corresponding weak$^*$ topology. This justifies the notation $\bar{\delta}^*_{\norm{\cdot}}$, which underlines the fact that this modulus depends on $\norm{\cdot}$ rather than on $\norm{\cdot}^*$. 
\end{remark}
It is well known that $(X, \norm{\cdot})$ is asymptotically uniformly smooth if and only if $(X^*,\norm{\cdot}^*)$ is weak$^*$-asymptotically uniformly convex. Therefore, when $X$ is reflexive, if we can find an equivalent norm $\abs{\cdot}$ on $X^*$ that is asymptotically uniformly smooth, then $\abs{\cdot}^*$ is an equivalent norm on $X^{**}=X$ that is asymptotically uniformly convex. Theorem \ref{thm:AUS} tells us that we need to show that $\Sz(X^*)\le \omega$ or, equivalently, that $\Sz(X^*,\vep)\in \N$ for all $\vep>0$. Since, for any Banach space $Y$, the map $\vep\mapsto \Sz(Y,\vep)$ is submultiplicative (see \cite{Lancien1995}) and non-increasing it is clearly sufficient to show that there is $\vep_0\in(0,1)$ such that $\Sz(X^*,\vep_0)\in \N$. It is not too difficult to show that for any Banach space $(Y,\norm{\cdot})$, if $\bar{\delta}^*_{\norm{\cdot}}(t_0)>0$, then $\Sz(Y,2t_0)\in \N$. Consequently, when $X$ is reflexive with $\bar{\delta}_{\norm{\cdot}}(t_0)=\bar{\delta}_{\norm{\cdot}^{**}}(t_0)=\bar{\delta}^*_{\norm{\cdot}^{*}}(t_0)>0$, it follows that $\Sz( X^*,2t_0)\in \N$. Recalling that $\bar{\delta}_{\norm{\cdot}}(t)\ge \bar{\beta}_{\norm{\cdot}}(\frac{t}{2})$, it follows from the discussion above that Theorem \ref{thm:main} holds for $t_0\in(0,\frac14)$. 

The first improvement needed to prove Theorem \ref{thm:main} for the whole open unit interval is rather elementary. 

\begin{proposition}
\label{prop:beta-delta}
Let $(X,\norm{\cdot})$ be a Banach space and $t_0\in (0,1)$. If $\bar{\beta}_{\norm{\cdot}}(t_0)>0$, then $\bar{\delta}_{\norm{\cdot}}(t)>0$, for all $t\in (t_0,1)$. 
\end{proposition}

\begin{proof} Let $t_0\in (0,1)$ and assume that $\bar{\beta}_{\norm{\cdot}}(t_0)>0$. Fix  $t\in (t_0,1)$ and assume, aiming for contradiction, that $\bar{\delta}_{\norm{\cdot}}(t)=0$. Pick $\eta >0$ such that 
\[\frac{t}{1+\eta}\ge t_0\ \ \text{and}\ \ \frac{1}{1+\eta}>1-\frac12 \bar{\beta}_{\norm{\cdot}}(t_0).\]
Since $\bar{\delta}_{\norm{\cdot}}(t)=0$, there exists $x\in S_X$ such that for all $Y \in \cof(X)$, there exists $x_Y \in S_Y$ such that $\|x + tx_Y\|\le 1+\eta$. We pick $x^*\in S_{X^*}$ such that $x^*(x)=1$. Then we build inductively a sequence  $(x_n^*)_{n=1}^\infty$ in $S_{X^*}$ such that for all $n\ge 1$, $x_n^*(x_{Y_n})=1$, where $Y_1 = \ker(x^*)$ and $Y_n = Y_1\cap \bigcap_{i=1}^{n-1} \ker(x_i^*)$ for $n\ge 2$. Next we denote $y_n=(1+\eta)^{-1}(x + tx_{Y_n})$. It readily follows from our construction and our choice of $\eta$ that $(y_n)_{n=1}^\infty$ is a $t_0$-separated sequence in $B_X$. We now use the definition of the $(\beta)$-modulus of $X$ to deduce that,  there exists $n>1$ such that 
$$\Big\|\frac{y_1+y_n}{2}\Big\|\le 1-\frac12\bar{\beta}_{\norm{\cdot}}(t_0).$$ 
But 
$$\Big\|\frac{y_1+y_n}{2}\Big\|\ge x^*\Big(\frac{y_1+y_n}{2}\Big)=\frac{1}{1+\eta}.$$
This is in contradiction with our choice of $\eta$. 
\end{proof}

Thanks to Proposition \ref{prop:comparesigma}, Theorem \ref{thm:main} now holds for all $t_0\in(0,\frac12)$. In order to gain another factor $2$ we need to introduce an ordinal index that is equivalent to the Szlenk index.
Consider the following derivation. For $K \subset X^*$ bounded and $\eps>0$, we set $\sigma_\eps(K)$ to be the set of all $x^*\in K$ such that there exists a net $(x_\alpha^*)_{\alpha \in A}$ in $K$ that is weak$^*$ converging to $x^*$ and so that $\norm{x^*-x^*_\alpha}> \eps$, for all $\alpha \in A$. Note that $x^*\in K\setminus \sigma_\eps(K)$ if and only if there exists a weak$^*$-open set $V$ such that $x^* \in V\cap K \subset B(x^*,\eps)$, where $B(x^*,\eps)$ denotes the norm-closed ball of center $x^*$ and radius $\eps$. If we define the pointed at $x^*$ radius of a set $S\subset X^*$ as follows
\begin{equation}
    \rad(x^*,S)=\sup_{y^*\in S} \|x^*-y^*\|,
\end{equation}
then 
\begin{align*}
\sigma_\vep(K) & = \{x^*\in K \colon \exists(x_\alpha^*)_{\alpha \in A} \textrm{ a net in } K \textrm{ with } x^*_\alpha \stackrel{w^*}{\to} x^* \textrm{ and } \norm{x^*-x^*_\alpha}> \eps, \forall \alpha\in A  \}  \\
              & = \{x^*\in K \colon \forall V\in \mathcal{V}_{w^*}(x^*),\, V\cap K\not\subset B(x^*,\vep)\}  \\
              & = \{x^*\in K \colon \forall V\in \mathcal{V}_{w^*}(x^*),\, \rad(x^*,V\cap K) > \vep\}.
\end{align*}
Then, we let $\sigma^0_\vep(K)=K$ and for any natural number $n\ge 0$, we define inductively $\sigma^{n+1}_\eps(K)=\sigma_\eps\big(\sigma^n_\eps(K)\big)$. In this note, we do not need to consider infinite ordinals. So we only set, if it exists, 
\begin{equation}
    \Sigma(X,\eps)=\inf\{n\in \N \colon\ \sigma_\eps^n(B_{X^*})=\emptyset\}
\end{equation}
and $\Sigma(X,\eps)=\infty$, otherwise. 

It is immediate from the definitions and an easy induction that $s_{2\eps}^n(B_{X^*}) \subset \sigma_\eps^n(B_{X^*}) \subset s_{\eps}^n(B_{X^*})$, and hence $\Sz(X,\vep)\in \N$ if and only if $\Sigma(X,\vep)\in\N$. The following proposition is then a direct consequence of classical facts about the Szlenk index and says that we can use the sigma-derivation to detect whether the Szlenk index is at most $\omega$.

\begin{proposition}\label{prop:comparesigma}  $\Sz(X)\le \omega$ if and only if $\Sigma(X,\eps)\in \N$, for all $\eps>0$. 
\end{proposition}

The sigma-derivation is designed to obtain the proposition below which avoids the loss of a factor $2$ which seems inherent when working with the Szlenk derivation.

\begin{proposition}
\label{prop:delta-sigma} 
 Let $(X,\norm{\cdot})$ be a Banach space. If there exists $t_0 \in (0,1)$ such that $\bar{\delta}^*_{\norm{\cdot}}(t_0)>0$, then $\Sigma(X,t_0) \in \N$. 
\end{proposition}

\begin{proof}
Let $x^*\in \sigma_{t_0}(B_{X^*})$. Then, there exists a net $(x^*_\alpha)_{\alpha\in A} \subset B_{X^*}$ weak$^*$-converging to $x^*$ and so that $\norm{x^*_\alpha-x^*}> t_0$, for all $\alpha \in A$. We claim that $\|x^*\|\le (1+\delta)^{-1}$, where $\delta=\bar{\delta}^*_{\norm{\cdot}}(t_0)$. So we may assume $x^*\neq 0$. Thus 
\begin{align*}
1&\ge \limsup_\alpha\|x^*_\alpha\|=\limsup_\alpha \|x^*\|\Big\|\frac{x^*}{\|x^*\|}+\frac{x^*_\alpha-x^*}{\|x^*\|}\Big\|\\
&\ge \|x^*\|\big(1+\bar{\delta}^*_X(\frac{t_0}{\|x^*\|})\big)\ge \|x^*\|(1+\delta).    
\end{align*}
We have shown that $\sigma_{t_0}(B_{X^*})\subset (1+\delta)^{-1}B_{X^*}$. By homogeneity, we deduce that for all $n\in \N$, $\sigma_{t_0}^n(B_{X^*})\subset (1+\delta)^{-n}B_{X^*}$. Therefore, there exists $n_0\in \N$ so that $\sigma_{t_0}^{n_0}(B_{X^*})\subset \frac{t_0}{2}B_{X^*}$ and then $\sigma_{t_0}^{n_0+1}(B_{X^*})=\emptyset$. 
\end{proof}

To be able to run an argument similar to the one discussed above for the Szlenk index, we need to verify that the sigma-derivation is submultiplicative and non-decreasing (and the latter is obvious).

\begin{lemma}
\label{lem:submult} 
Let $\eps,\eps' \in (0,1)$ and assume that $\Sigma(X,\eps)=n\in \N$. Then, for all $k \in \N$,
\begin{equation}
\sigma_{\eps\eps'}^{nk}(B_{X^*})\subset \sigma_{\eps'}^k(B_{X^*}).
\end{equation}
In particular, 
\begin{equation*}
    \Sigma(X,\vep\vep')\le \Sigma(X,\vep)\Sigma(X,\vep').
\end{equation*}
\end{lemma}

\begin{proof}
We will show this by induction on $k$. This is clearly true for $k=0$. So assume that for some $k \ge 0$, $\sigma_{\eps\eps'}^{nk}(B_{X^*})\subset \sigma_{\eps'}^k(B_{X^*})$. Let $x^* \in B_{X^*}\setminus \sigma_{\eps'}^{k+1}(B_{X^*})$. We want to show that $x^* \notin \sigma_{\eps\eps'}^{n(k+1)}(B_{X^*})$, so, by the induction hypothesis, we may as well assume that $x^* \in \sigma_{\eps\eps'}^{nk}(B_{X^*})\subset \sigma_{\eps'}^k(B_{X^*})$. Therefore, there exists a weak$^*$-open set $V$ such that $x^*\in V \cap \sigma_{\eps'}^k(B_{X^*}) \subset B(x^*,\eps')$. 
We claim that 
\begin{equation}\label{eq:subinduction}
\forall l \in \{0,\ldots,n\},\ \ V\cap \sigma_{\eps\eps'}^{nk+l}(B_{X^*}) \subset \sigma_{\eps\eps'}^l(B(x^*,\eps')).    
\end{equation}    
The lemma follows from \eqref{eq:subinduction} as follows. Since by definition of $n$, $\sigma_\eps^n(B_{X^*})=\emptyset$, an easy homogeneity and translation argument implies that $\sigma_{\eps\eps'}^n(B(x^*,\eps'))=\emptyset$. Now, observe that by taking $l=n$ in \eqref{eq:subinduction} we have $V\cap \sigma_{\eps\eps'}^{n(k+1)}(B_{X^*})=\emptyset$. In particular, $x^* \notin \sigma_{\eps\eps'}^{n(k+1)}(B_{X^*})$. 

It remains to prove \eqref{eq:subinduction} and this can be done via an induction on $l$. The inclusion in \eqref{eq:subinduction} certainly holds for $l=0$, and assume that it holds for some $0\le l <n$. Let $y^*\notin \sigma_{\eps\eps'}^{l+1}(B(x^*,\eps'))$. We want to show that $y^* \notin V\cap \sigma_{\eps\eps'}^{nk+l+1}(B_{X^*})$, so we may assume that $y^* \in V\cap \sigma_{\eps\eps'}^{nk+l}(B_{X^*})\subset \sigma_{\eps\eps'}^l(B(x^*,\eps'))$. Thus, there exists $W$ weak$^*$-open so that 
$$y^*\in W \cap \sigma_{\eps\eps'}^l(B(x^*,\eps')) \subset B(y^*,\eps\eps').$$
It follows that 
$$y^* \in V\cap W \cap \sigma_{\eps\eps'}^{nk+l}(B_{X^*})\subset W \cap \sigma_{\eps\eps'}^l(B(x^*,\eps')) \subset B(y^*,\eps\eps').$$
Since $V \cap W$ is weak$^*$ open, we deduce that $y^* \notin \sigma_{\eps\eps'}^{nk+l+1}(B_{X^*})$. This finishes the proof of the claim and hence of the lemma. 
\end{proof}

\begin{corollary}
\label{cor:sigma-aus} 
Let $X$ be a Banach space. If there exists $\eps_0 \in (0,1)$ such that $\Sigma(X,\eps_0) \in \N$, then $X$ admits an equivalent norm that is asymptotically uniformly smooth.
\end{corollary}

\begin{proof}
It follows from Lemma \ref{lem:submult} that for all $n\in \N$, $\Sigma(X,\eps_0^n)\le (\Sigma(X,\eps_0))^n$. Since $\Sigma(X,.)$ is clearly non-increasing, we deduce that $\Sigma(X,\eps) \in \N$, for all $\eps>0$, and thus, by Proposition \ref{prop:comparesigma} that $\mathrm{Sz}(X)\le \omega$, which implies that $X$ is AUS renormable.
\end{proof}

We are now able to prove the theorem below for the whole open unit interval.

\begin{theorem}
\label{thm:beta->AUCable}
Let $(X,\norm{\cdot})$ be a Banach space. If there exists $t_0\in (0,1)$ such that $\bar{\beta}_{\norm{\cdot}}(t_0)>0$, then $X$ is reflexive and admits an equivalent norm that is asymptotically uniformly convex.     
\end{theorem}

\begin{proof}
Let $t_0\in(0,1)$ such that $\bar{\beta}_{\norm{\cdot}}(t_0)>0$. The reflexivity of $X$ is ensured by Proposition \ref{prop:beta-reflexive}. By Proposition \ref{prop:beta-delta}, for any $t_0<t_1<1$, $\bar{\delta}_{\norm{\cdot}}(t_1)>0$, and since $X$ is reflexive we have that $\bar{\delta}^*_{\norm{\cdot}^{*}}(t_1)>0$. Invoking Proposition \ref{prop:delta-sigma} it follows that $\Sigma( X^*,t_1)\in \N$ and we deduce from Corollary \ref{cor:sigma-aus} that $X^*$ is AUS renormable and thus, invoking again reflexivity, $X$ is AUC renormable.
\end{proof}

We can now conclude the proof of our main result.

\begin{proof}[End of proof of Theorem \ref{thm:main}]
We know from Proposition \ref{prop:beta-reflexive} and Theorem \ref{thm:beta->AUCable} that $X$ is reflexive and AUC renormable. We also know from Theorem \ref{thm:beta->AUS} that $X$ is AUS renormable. It follows from classical results that $X$ admits an equivalent norm that is both AUC and AUS (see \cite{Prus1989}, \cite{OdellSchlumprecht2002} and \cite{DKLR2017}). Since $X$ is reflexive, this implies that this norm has property $(\beta)$ (cf \cite[Theorem 4]{Kutzarova1990}.     
\end{proof}

\section{Concluding remarks}

It is elementary to come up with an example showing that the converse of Theorem \ref{thm:main} does not hold. This construction is somehow dual to the example of a Banach space with property $(\beta)$ that is not asymptotically uniformly smooth given in \cite[Example 5]{Kutzarova1990}.
\begin{example} 
Let $p\in (1,\infty)$ and let $\norm{\cdot}_p$ denote the natural norm on $\ell_p$. Consider the Banach space $X=\ell_p \oplus_\infty \ell_p$ whose norm is given by
\begin{equation*}
(x,y)\in X \mapsto \norm{(x,y)}= \max\{\norm{x}_p,\norm{y}_p\}.
\end{equation*}
Obviously, $(X,\norm{\cdot})$ is isomorphic to $\ell_p$ and therefore admits an equivalent norm with property $(\beta)$. Let us denote $(e_n)_{n=1}^\infty$ the canonical basis of $\ell_p$, and for $n\in \N$, $z_n=(e_1,e_n)$. Then $(z_n)_{n=1}^\infty$ is a $2^{1/p}$-separated sequence in the unit ball of $X$. However, we have that for all $n \in \N$:
\begin{equation*}
\Big\|\frac{z_1+z_n}{2}\Big\|=\max\Big\{\norm{e_1}_p, \bnorm{\frac{e_1+e_n}{2}}_p\Big\}=1.
\end{equation*}
This shows that $\bar\beta_{\norm{\cdot}}(2^{1/p})=0.$ Therefore, for all $t\in (0,2)$ there exists a Banach space $(X,\norm{\cdot})$ such that $\bar\beta_{\norm{\cdot}}(s)=0$ for all $s\in (0,t]$, but $X$ admits an equivalent norm with property $(\beta)$. 
\end{example}

As we will see, passing to an equivalent norm is necessary in Theorem \ref{thm:main} as it is not true that a Banach space $(X,\norm{\cdot})$ for which $\bar\beta_{\norm{\cdot}}(t_0)>0$ for some $t_0\in (0,1)$ has property $(\beta)$. In fact, one could ask a similar question regarding Theorem \ref{thm:James-Enflo}. In this case, it is easy to see that by creating flat spots on the unit sphere of the Euclidean plane we obtain a norm that is uniformly non-square but not uniformly convex. Formally, if we consider the norm $\abs{\cdot}$ on $\ell_2^2$ whose unit ball is the closed convex hull of the set $B_{\ell_2^2}\cup \{\pm(e_1 + \frac12 e_2), \pm(e_1 + \frac12 e_2)\}$, then clearly $\delta_{\abs{\cdot}}(t)>0$ for all $t\in(1,2]$ and $\delta_{\abs{\cdot}}(t)=0$ for all $t\in[0,1]$. 
A similar ``flattening'' of the unit ball of $\ell_2$ provides, for any $t_0 \in (0,1)$, an example of a Banach space $(X,\norm{\cdot})$ such that $\bar\beta_{\norm{\cdot}}(t_0)>0$, but the norm $\norm{\cdot}$ does not have property $(\beta)$. 
\begin{example}
Let $(e_n)_{n=1}^\infty$ be the canonical basis of $\ell_2$ and fix $\eps \in (0,1)$. We define $\abs{\cdot}_\eps$ to be the norm on $\ell_2$, whose unit ball is the closed convex hull of $B_{\ell_2}\cup\bigcup_{n=2}^\infty\{\pm(e_1+\eps e_n),\pm(e_1-\eps e_n)\}$. It is easily checked, by comparing the unit balls, that for all $x \in \ell_2$,
\begin{equation}
\label{equivalence}
\abs{x}_\eps \le \norm{x}_2 \le (1+\eps^2)^{1/2} \abs{x}_\eps.    
\end{equation}
Note also that $\abs{e_n}_\eps=1$ for all $n\in \N$ and $\abs{e_1 + te_n}_\eps=1$ for all $t \in [0,\eps]$. 
For $n\ge 2$, define $y_n=e_1+\eps e_n$ and observe that an elementary computation reveals that $(y_n)_{n=2}^\infty$ is an $\eps$-separated sequence in the unit ball of $\abs{\cdot}_\eps$. On the other hand, for $n\ge 2$, $\frac12(e_1+y_n)=e_1+\frac{\eps}{2}e_n$ and therefore, $\abs{\frac12(e_1+y_n)}_\eps=1$. This shows that $\bar{\beta}_{\abs{\cdot}_\eps}(\eps)=0$ and that $\abs{\cdot}_\eps$ fails property $(\beta)$. Let us now fix $t_0 \in (0,1)$. Using the fact that $\bar{\beta}_{\norm{\cdot }_2}(t)>0$ for all $t\in(0,1)$ and the inequalities \eqref{equivalence}, it is then clear that for $\eps\in(0,t_0)$ small enough, $\bar{\beta}_{\abs{\cdot}_\eps}(t_0)>0$, and yet $\bar{\beta}_{\abs{\cdot}_\eps}(\vep)=0$.
\end{example}

The natural domain for the modulus of uniform convexity is clearly $[0,2]$. Since every infinite-dimensional Banach space has a $1$-separated sequence in its unit sphere it is natural to consider the interval $[0,1]$ as domain of the modulus of property $(\beta)$. However, it also makes sense to consider this modulus on the interval $[0,K(X))$ where $K(X)$ is the Kottman's constant of $X$, i.e.,
\begin{equation}
    K(X) :=\sup \{ \lambda>0 \colon \exists (x_n)_{n\in \N} \subset B_X \textrm{ s.t. } \inf_{n\neq m}\norm{x_n -x_m}\ge \lambda\}.
\end{equation}

The non-trivial fact that $K(X)>1$ is due to Elton and Odell \cite{EltonOdell1981}.
It would be interesting to see if the condition in Theorem \ref{thm:main} could be pushed beyond 1.


\end{document}